  \documentclass[11pt]{article}
\usepackage{amssymb, amsthm, amsmath, amscd}
\usepackage[all]{xy}
\usepackage{mathrsfs}

\newtheorem{thm}{Theorem}[section]
\newtheorem{lem}[thm]{Lemma}

\newtheorem{NN}[thm]{}
\theoremstyle{definition}\newtheorem{df}[thm]{Definition}
\theoremstyle{definition}
\theoremstyle{definition}

\renewcommand{\phi}{\varphi}

\newcommand{\Q}{\mathbb{Q}}

\newcommand{\T}{\mathbb{T}}

\newcommand{\hm}{homomorphism}
\newcommand{\dt}{\delta}
\newcommand{\ep}{\epsilon}
\newcommand{\andeqn}{\,\,\,{\rm and}\,\,\,}
\newcommand{\rforal}{\,\,\,{\rm for\,\,\,all}\,\,\,}
\newcommand{\CA}{$C^*$-algebra}
\newcommand{\SCA}{$C^*$-subalgebra}

\newcommand{\beq}{\begin{eqnarray}}
\newcommand{\eneq}{\end{eqnarray}}
\newcommand{\tforal}{\,\,\,\text{for\,\,\,all}\,\,\,}

\newcommand{\fp}{{\mathfrak{p}}}

\title{Inductive Limits of Subhomogeneous $C^*$-algebras with Hausdorff Spectrum}
\author{Huaxin Lin}
\date{}

\begin{document}

\maketitle

\begin{abstract}
 We consider unital simple inductive limits of
 generalized dimension drop \CA s
 They are so-called ASH-algebras and  include all unital simple AH-algebras and all
dimension drop \CA s. Suppose that $A$ is one of these \CA s. We show
that $A\otimes Q$ has tracial rank no more than one, where $Q$ is
the rational UHF-algebra. As a consequence, we obtain the
following classification result:  Let $A$ and $B$ be  two unital
simple inductive limits of generalized dimension drop algebras with no
dimension growth.  Then $A\cong B$ if and only if they have the
same Elliott invariant. 

\end{abstract}

\section{Introduction}

A unital AH-algebra is an inductive limit of finite direct sums of
\CA s with the form
$$
PC(X, M_k)P,
$$
where $X$ is a finite dimensional compact metric space, $k\ge 1$
is an integer and $P\in C(X, M_k)$ is a projection. One of the
successful stories in the program of classification of amenable
\CA s, or otherwise known as the Elliott program is the
classification of unital simple AH-algebras with no dimension
growth (see \cite{EGL2}) by their $K$-theoretical invariant, or
otherwise called the Elliott invariant. Other classes of inductive
limits of building blocks studied in the program include so-called
simple inductive limits of dimension drop interval (or circle)
algebras (see \cite{S}, \cite{EGJS}, \cite{Ts}, \cite{JS} and \cite{M}, for
example). One reason to study the simple inductive limits of
dimension drop interval algebras is to construct unital simple \CA
s with the Elliott invariant which can not be realized by unital
simple AH-algebras. Notably, the so-called Jiang Su-algebras of
unital projectionless simple ASH-algebra ${\cal Z}$ whose Elliott
invariant is the same as that of complex field.

However, until very recently, they were classified as separate
classes of simple \CA s. In particular, it was not known that one
\CA\,  in one of these classes should be isomorphic to ones in a
different  class with the same Elliott invariant in general.
Furthermore, methods used in classification of unital simple AH-algebras 
and unital simple dimension drop algebras dealt with maps from 
dimension drop algebras to dimension drop algebras and maps between 
some homogeneous \CA s. These methods do not work for maps from 
dimension drop algebras into homogeneous \CA s or maps from homogeneous 
\CA s into dimension drop algebras in general. Therefore most classification theorems for simple inductive limits 
of basic building blocks  do not deal with mixed building blocks. 
In
this paper we consider a general class of inductive limits of
sub-homogenous \CA s with Hausdorff spectrum which includes both
AH-algebras and all known unital simple inductive limits of
dimension drop algebras as well as inductive limits with mixed basic building blocks. We will give a classification theorem for
this class of unital simple \CA s.

Let $X$ be a compact metric space. Let $k\ge 1$ be an integer, let
$\xi_1,\xi_2,...,\xi_n\in X$ and let $m_1, m_2,...,m_n\ge 1$ be
integers such that $m_j|k,$ $j=1,2,...,n.$ One may write
$M_k=M_{m_j}\otimes M_{k/m(j)}.$ Let  $B_j\cong M_{m_j}\otimes
1_{k/m(j)}$ be a unital \SCA\, of $M_k$ ($1_{B_j}=1_{M_k}$).

\begin{df}\label{D1}
Define
$$
D_k(X, \xi_1,\xi_2,...,\xi_n, m_1,m_2,....,m_n)=\\
\{f\in C(X, M_k): f(\xi_j)\in B_j,\,1\le j\le n\}.
$$
\end{df}


In the above form, if $X$ is an interval, it is known as (general)
dimension drop interval algebra, and if $X=\T,$ it is known as
(general) dimension drop circle algebra. Here we allow high dimensional 
dimension drop algebras. So when $X$ is a
connected finite CW complex, algebras $D_k(X,
\xi_1,\xi_2,...,\xi_n, m_1,m_2,....,m_n)$ and finite direct sums
of their  unital hereditary \SCA s may be called general dimension
drop algebras. When $m_j=k,$ $1\le j\le d(j),$ $j=1,2,...,N,$ it
is a homogenous \CA. Thus AH-algebras, inductive limits of
dimension drop interval algebras and inductive limits of dimension
drop circle algebras are inductive limit of general dimension drop
algebras. A classification theorem for unital simple inductive
limits of those general dimension drop algebras would not only
generalize many previous known classification theorems but more
importantly it would unify these classification results. However,
one may consider more general \CA s.


Let $X_1, X_2, ...,X_n\subset X$ be disjoint compact subsets. Let
$k,$ $m_1, m_2,...,m_n$ and $B_1,B_2,...,B_n$ be as above. Define
\begin{df}
\beq\label{D2}
{\bar D}_{X,k, \{X_j\}, n}= \{f\in C(X, M_k): f(x)\in B_j\tforal
x\in X_j, \, 1\le j\le n\}.
\eneq
\end{df}
By a generalized dimension drop algebra, we mean
 \CA s of the form:
$$
\bigoplus_{j=1}^N P_j{\bar D}_{X_j,r(j),\{X_{j,i}\},d(j)}P_j
$$
where $P_j\in D_{X_j, r(j),\{X_{i,j}\}, d(j)}$ is a projection,
$X_j$ is a compact metric space with finite dimension and with
property (A) (see \ref{DA} bellow) and $\{X_{j,i}\}$ is a
collection of finitely many disjoint compact subsets of $X_j.$

Note that, in the following, $A\otimes
{\cal Z}$ and $A$ have exactly the same Elliott invariant whenever $K_0(A)$ is weakly unperforated.

 We prove the following theorem:
\begin{thm}\label{MT1}
Let $A$ and let $B$ be two unital simple inductive limits of
generalized dimension drop algebras. Then $A\otimes {\cal Z}\cong
B\otimes {\cal Z}$ if and only if
$$
(K_0(A), K_0(A)_+, [1_A], K_1(A), T(A), r_A)\cong (K_0(B),
K_0(B)_+, [1_B], K_1(B), T(B), r_B).
$$
\end{thm}
(see \ref{DK}).

Denote by ${\cal D}$ the class of inductive limits of generalized
dimension drop algebras.
By a recent result of W. Winter (\cite{W3}), unital simple
ASH-algebras with no dimension growth are ${\cal Z}$-stable.
Therefore we also have the following:

\begin{thm}\label{MT2}
Let $A,\, B\in {\cal D}$ be two unital simple \CA s with no
dimension growth. Then $A\cong B,$ if and only if
$$
(K_0(A), K_0(A)_+, [1_A], K_1(A), T(A), r_A)\cong (K_0(B),
K_0(B)_+, [1_B], K_1(B), T(B), r_B).
$$
\end{thm}

This unifies the  previous known classification theorems for unital
simple ASH-algebras with Hausdorff spectrum such as 
\cite{EGL2}, \cite{EGJS}, \cite{JS}, \cite{Ts} and \cite{M}, etc.
In fact, we show that, if $A\in {\cal D}$ is a unital simple \CA,
then the tracial rank of $A\otimes Q$ is no more than one, where
$Q$ is the UHF-algebra with $(K_0(Q), K_0(Q)_+, [1_Q])=(\Q, \Q_+,
1).$ We then obtain the above mentioned classification theorems by  applying a recent classification theorem in
\cite{Lnclasn} for unital simple \CA s $A$  with the  tracial rank of $A\otimes Q$  no more than one. 
Therefore one may view our result here as an application of the classification result in \cite{Lnclasn}.

\vspace{0.2in}

{\bf Acknowledgments}: This work was mostly done when the author
was in East China Normal University during the summer of 2008. It
was partially supported by an NSF grant (DMS074813), the
Chang-Jiang Professorship from East China Normal University and
Shanghai Priority  Academic Disciplines.

\section{Preliminaries}

\begin{NN}
{\rm Let $\{A_n\}$ be a sequence of \CA s and let $\phi_n: A_n\to
A_{n+1}$ be \hm s. Denote by $A=\lim_{n\to\infty}(A_n, \phi_n)$
the \CA\, of the inductive limit. Put $\phi_{n, m}=\phi_m\circ
\phi_{m-1}\circ \cdots \phi_n: A_n\to A_m$ for $m>n.$ Denote by
$\phi_{n,\infty}: A_n\to A$ the \hm\, induced by the inductive
limit system.}

\end{NN}

\begin{NN} {\rm
Let $X$ be a compact metric space, let $\xi\in X$ and let $d>0.$
Set
$$
B(\xi, d)=\{x\in X: {\rm dist}(x, \xi)<d\}.
$$

Suppose that $0<d_2<d_1.$ Set
$$
A(\xi, d_1,d_2)=\{x\in X: d_1>{\rm dist}(x, \xi)>d_2\}.
$$
}
\end{NN}

\begin{df}
{\rm Let $X$ be a compact metric space, let $A$ be a \CA\,  and
let $B\subset C(X, A)$ be a \SCA. For each $x\in X,$ denote by
$\pi_x: C(X, A)\to A$  the \hm\, defined by $\pi_x(f)=f(x)$ for
all $f\in C(X,A).$ Suppose that $C_x=\pi_x(B).$ Then $\pi_x: B\to
C$ is a \hm. }

\end{df}

\begin{NN}
{\rm Let $A$ be a \CA,  let $x,\,y\in A$ be two elements and let
$\ep>0.$ We write
$$
x\approx_{\ep} y,
$$
if $\|x-y\|<\ep.$

Suppose that  ${\cal S}\subset A$ is a subset. We write
$$
x\in_{\ep} {\cal S},
$$
if $x\approx_{\ep} y$ for some $y\in {\cal S}.$

Let $B$ be another \CA, let $\phi, \psi: A\to B$ be two maps, let
${\cal S}\subset A$ be a subset and let $\ep>0.$ We write
$$
\phi\approx_{\ep} \psi \,\,\,{\rm on}\,\,\, {\cal S},
$$
if $\psi(x)\approx_{\ep} \phi(x)$ for all $x\in {\cal S}.$

}

\end{NN}

\begin{NN}
{\rm  Let $A$ be a unital \CA. Denote by $U(A)$ the unitary group
of $A.$ Denote by $U_0(A)$ the connected component of $U(A)$
containing the identity.}

\end{NN}

\begin{NN}
{\rm

Let $A$ be a unital \CA. We say $B$ is a unital \SCA\, of $A$  if
$B\subset A$ is a unital \CA\, {\em and $1_B=1_A.$}

}

\end{NN}

\begin{df}
Let $A$ be a unital \CA\, and let $C$ be another \CA. Suppose that
$\phi_1, \phi_2: C\to A$ are two \hm s. We say that $\phi_1$ and
$\phi_2$ are asymptotically unitarily equivalent if there exists a
continuous path of unitaries  $\{u(t): t\in [0, \infty)\}\subset
A$ such that
\beq\label{Dasu-1}
\lim_{t\to\infty}u(t)^*\phi_2(c)u(t)=\phi_1(c)\tforal c\in C.
\eneq

We say that $\phi_1$ and $\phi_2$ are strongly asymptotically
unitarily equivalent if in (\ref{Dasu-1}) $u(0)=1_A.$  Of course,
$\phi_1$ and $\phi_2$ are strongly asymptotically unitarily
equivalent if they are asymptotically unitarily equivalent and if
$U(A)=U_0(A).$

\end{df}

\begin{df} Denote by ${\cal I}$ the class of those \CA s with the
form $\bigoplus_{i=1}^n M_{r_i}(C(X_i)),$ where each $X_i$ is a
finite CW complex with (covering) dimension no more than one.

A unital simple \CA\, $A$ is said to have tracial rank no more
than  one (we write $TR(A)\le 1$)  if for any finite subset
$\mathcal F\subset A$, $\ep>0$, any nonzero positive element $a\in
A$, there is a \SCA\,  $C\in\mathcal I$ such that if denote by $p$
the unit of $C$, then for any $x\in\mathcal F$, one has
\begin{enumerate}
\item $\|xp-px\|\leq\ep,$ \item there is $b\in C$ such that
$\|b-pxp\|\leq \ep$  and \item $1-p$ is Murray-von Neumann
equivalent to a projection in $\overline{aAa}$.
\end{enumerate}
\end{df}

\begin{df}
For a supernatural number $\mathfrak p$, denote by $M_\fp$ the UHF
algebra associated with $\fp$ (see \cite{Dix}). We use $Q$
throughout the paper for the UHF-algebra with $(K_0(Q), K_0(Q)_+,
[1_Q])=(\Q, \Q_+, 1).$ We may identify $Q$ with the inductive
limit $\lim_{n\to\infty}(M_{n!},h_n),$ where $h_n: M_{n!}\to
M_{(n+1)!}$ is defined by $a\mapsto a\otimes 1_{n+1}$ for $a\in
M_{n!}.$

Let ${\cal A}$ be the class of all unital separable simple
amenable \CA s $A$ in ${\cal N}$ for which $TR(A\otimes
M_{\mathfrak{p}})\le 1$ for some  supernatural number
${\mathfrak{p}}$ of infinite type.
\end{df}

\begin{df}\label{DK}
Let $A$ be a unital stably finite separable simple amenable \CA.
Denote by $T(A)$ the tracial state space of $A.$ We also use
$\tau$ for $\tau\otimes {\rm Tr}$ on $A\otimes M_k$ for any
integer $k\ge 1,$ where ${\rm Tr}$ is the standard trace on $M_k.$

By $\mathrm{Ell}(A)$ we mean the following:
$$
(K_0(A), K_0(A)_+, [1_A], K_1(A), T(A), r_A),
$$
where $r_A: T(A)\to S_{[1_A]}(K_0(A))$ is a surjective continuous
affine map such that $r_A(\tau)([p])=\tau(p)$ for all projections
$p\in A\otimes M_k,$ $k=1,2,...,$ where $S_{[1_A]}(K_0(A))$ is the
state space of $K_0(A).$

Suppose that $B$ is another stably finite unital separable simple
\CA. A map $\Lambda: \mathrm{Ell}(A)\to \mathrm{Ell}(B)$ is said
to be a \hm\, if $\Lambda$ gives an order \hm\, $\lambda_0:
K_0(A)\to K_0(B)$ such that $\lambda_0([1_A])=[1_B],$ a \hm\,
$\lambda_1: K_1(A)\to K_1(B),$ a continuous affine map
$\lambda_{\rho}': T(B)\to T(A)$ such that
$$
\lambda_{\rho}'(\tau)(p)=r_B(\tau)(\lambda_0([p]))
$$
for all projection in $A\otimes M_k,$ $k=1,2,...,$ and for all
$\tau\in T(B).$

We say that such $\Lambda$ is an isomorphism, if $\lambda_0$ and
$\lambda_1$ are isomorphisms and $\lambda_\rho'$ is a affine
homeomorphism. In this case, there is an affine homeomorphism
$\lambda_\rho: T(A)\to T(B)$ such that
$\lambda_\rho^{-1}=\lambda_\rho'.$

\end{df}

\begin{thm}[Corollary 11.9 of \cite{Lnclasn}]\label{MTQ}
Let $A,\, B\in {\cal A}.$ Then
$$
A\otimes {\cal Z}\cong B\otimes {\cal Z}
$$
if
$$
\mathrm{Ell}(A\otimes {\cal Z})=\mathrm{Ell}(B\otimes {\cal Z}).
$$
\end{thm}

\section{Approximately unitary equivalence}

The following is known.

\begin{lem}\label{asy}
Let $\phi, \psi: Q\to Q$ be two unital \hm s. Then they are
strongly asymptotically unitarily equivalent.

\end{lem}

\begin{proof}
Since $\phi$ and $\psi$ are unital, one computes that $\phi$ and
$\psi$ induce the same identity map on $K_0(Q).$ Since $K_0(Q)$ is
divisible and $K_1(Q)=\{0\},$ one then checks that
$$
[\phi]=[\psi]\,\,\,{\rm in}\,\,\, KK(Q,Q).
$$
With $K_1(Q)=\{0\},$  it follows from Corollary 11. 2 of
\cite{Lnasym} that $\phi$ and $\psi$ are strongly asymptotically
unitarily equivalent.

\end{proof}

\begin{lem}\label{almostcommut}
For any $\ep>0$ and any finite subset ${\cal F}\subset Q,$ there
exists $\dt>0$ and a finite subset ${\cal G}\subset Q$ satisfying
the following: suppose that $u\in U(Q)$ and $\phi: Q\to Q$ is a
unital \hm\, such that

\beq \|[u, \, \phi(b)]\|<\dt\tforal b\in {\cal G}. \eneq

Then there exists a continuous path of unitaries $\{u(t): t\in
[0,1]\}$ in $Q$ such that $u(0)=u,$ $u(1)=1,$

\beq
\|[u(t),\, \phi(a)]\|<\ep\tforal a\in {\cal F}.
\eneq

Moreover,
$$
{\rm Length}(\{u(t)\})\le 2\pi.
$$

\end{lem}

\begin{proof}
Note that $K_1(Q)=\{0\}$ and $K_i(Q\otimes C(\T))\cong \Q$ is
torsion free divisible group,\, $i=0,1.$  One computes that
$$
{\rm Bott}(\phi,u)=0.
$$
It follows from \cite{Lnhomp} that such $\{u(t)\}$ exists.

\end{proof}

%
%


\begin{thm}\label{Appr}
Let $X$ be a compact subset of $[0,1]$ with $a=\inf\{t:t\in X\}$
and $b=\sup\{t:t\in X\}.$ Suppose that $\phi, \psi: Q\to C(X, Q)$
are two unital monomorphisms. Then $\phi$ and $\psi$ are
approximately unitarily equivalent. Moreover, if $\pi_a\circ
\phi=\pi_a\circ \psi$ (or $\pi_b\circ \phi=\pi_b\circ\psi,$ or
both hold) then there exists a sequence of unitaries $\{u_n\}\in
C(X, Q)$ such that $\pi_a(u_n)=1$ {\rm (}or $\pi_b(u_n)=1,$ or
$\pi_a(u_n)=\pi_b(u_n)=1)$ and
$$
\lim_{n\to\infty}{\rm ad} u_n\circ \phi(x)=\psi(x)\tforal x\in Q.
$$

Furthermore, if $F\subset Q$ is a full matrix algebra as a unital
\SCA, $\pi_a\circ \phi|_F=\pi_a\circ \psi|_F$ and $\pi_b\circ
\phi|_F=\pi_b\circ \psi|_F,$ one can find, for each $\ep>0,$  a
unitary $u$ such that
$$
\|{\rm ad}\, u\circ \phi(f)-\psi(f)\|<\ep
$$
for all $f\in F$ and $\pi_a(u)=\pi_b(u)=1.$

\end{thm}

\begin{proof}
To simplify the notation, without loss of generality, we may
assume that $a=0$ and $b=1.$

 Fix $\ep>0$ and a finite subset ${\cal F}.$ Without loss of generality, we may assume that
 $\|f\|\le 1$ for all $f\in {\cal F}.$
Let $\dt>0$ and ${\cal G}$ be the finite subset required by
\ref{almostcommut} corresponding to $\ep$ and ${\cal F}.$ Without
loss of generality, we may assume that $\ep<\dt$ and ${\cal
F}\subset {\cal G}.$

Let $\phi, \psi: Q\to C(X, Q)$ be two  unital \hm s.  Denote by
$\pi_t: C(X, Q)\to Q$ the point-evaluation map at the point $t$
($\in X$). There is a partition:
$$
0=t_0<t_1\cdots <t_n=1
$$
with $t_i\in X$ such that either $(t_{i-1}, t_i)\cap X=\emptyset,$
or

\beq\label{T1-1}
\|\pi_{t_i}\circ \phi(b)-\pi_t\circ \phi(b) \|<\dt/16\andeqn
\|\pi_{t_i}\circ \psi(b)-\pi_t\circ \psi(b) \|<\dt/16
\eneq
for all $t\in [t_{i-1}, t_i]\cap X,$ $i=1,2,...,n$ and all $b\in
{\cal G}.$
Since $K_0(Q)=\Q,$ $[1_Q]=1$ and $K_1(Q)=\{0\},$ one computes that
$$
[\pi_t\circ \phi]=[\pi_t\circ \psi]\,\,\,{\rm in}\,\,\,
KK(Q,Q)\,\,\,\tforal t\in X.
$$
It follows from \ref{asy} that there is, for each $i,$ a unitary
$v_i\in Q$ such that
\beq\label{T1-2}
 \|{\rm ad}\, v_i\circ \pi_{t_i}\circ
\phi(b)-\pi_{t_i}\circ \psi(b)\|<\dt/16, \eneq $i=0,1,...,n.$ Put
$w_i=u_{i-1}u_i^*,$ if $(t_{i-1}, t_i)\cap X\not=\emptyset.$ Then,
if $(t_{i-1}, t_i)\cap X\not=\emptyset,$ \beq
w_i^*(\pi_{t_{i-1}}\circ \phi(b))w_i&\approx_{\dt/16}&
u_i(\pi_{t_{i-1}}\circ\psi(b))u_i^*\\
&\approx_{\dt/16}& u_i(\pi_{t_i}\circ \psi(b))u_i^*\\
&\approx_{\dt/16}& \pi_{t_i}\circ \phi(b)\\
&\approx_{\dt/16} & \pi_{t_{i-1}}\circ \phi(b)
\eneq
for all $b\in {\cal G}.$ It follows from \ref{almostcommut} that
there exists a continuous path of unitaries $\{w_i(t):t\in
[t_{i-1},t_i]\}$ in $Q$ such that $w_i(t_{i-1})=w_i,$ $w_i(t_i)=1$
and \beq\label{T1-3} \|[w_i(t),\, \pi\circ
\pi_{t_{i-1}}(a)]\|<\ep/16 \eneq for all $t\in [t_{i-1}, t_i]\cap
X$ and $a\in {\cal F},$ $i=1,2,...,n.$ Define $u(t_i)=u_i$ and (if
$(t_{i-1}, t_i)\cap X\not=\emptyset$) 
\beq\label{T1-4}
 u(t)=w_i(t)u_i\rforal t\in [t_{i-1},
t_i]\cap X
 \eneq
 Note
 that $u(t)$ is
continuous. Moreover, when $(t_{i-1}, t_i)\cap X\not=\emptyset,$
for $t\in [t_{i-1}, t_i]\cap X,$
\beq 
{\rm ad}\, u(t)\circ \pi_t\circ \phi(a) &\approx_{\ep/16}&
{\rm ad} \, u(t)\circ \pi_{t_{i}}\circ \phi(a)\\
&\approx_{\ep/16}& {\rm ad}\, u_i\circ \pi_{t_i}\circ \phi(a)
\\
&\approx_{\ep/16}& \pi_{t_i}\circ \psi(a)\\
&\approx_{\ep/16}&\pi_t\circ \psi(a)
 \eneq
for all $a\in {\cal F},$ where the first estimate follows from
(\ref{T1-1}), the second estimate follows from (\ref{T1-3}), the
third follows from (\ref{T1-2}) and the last one follows from
(\ref{T1-1}).  Note that $u(t)\in C(X, Q).$ It follows that $\phi$
and $\psi$ are approximately unitarily equivalent.

To prove the second part of the statement, we use what we have just
proved.
Let $\ep>0$ and ${\cal F}\subset Q$ be a finite subset. Let $\dt$
and ${\cal G}$ be required by \ref{almostcommut} for $\ep/4$ and
${\cal F}.$
Without loss of generality, we may assume that ${\cal F}\subset
{\cal G}.$
Put $\ep_1=\min\{\dt/4, \ep/4\}.$ There exists $\eta>0$ such that
\beq
\| \pi_t\circ \phi(a)-\pi_{t'}\circ \phi(a)\|&<&\ep_1/2\,\,\, \andeqn\\
\|\pi_t\circ \psi(a)-\pi_{t'}\circ \psi(a)\|&<&\ep_1/2 \eneq for
all $a\in {\cal G},$ whenever $|t-t'|\le \eta$ and $t, t'\in X.$

From what we have shown, there exists a unitary $w\in C(X, Q)$
such that
$$
\|{\rm ad}\, w\circ \psi(a)-\phi(a)\|<\ep_1/2\tforal a\in {\cal
G}.
$$

If $0$ is an isolated point in $X,$ then we may assume that
$\pi_0(w)=1.$ Otherwise, let
$$
\eta_1=\sup\{t: t\in (0, \eta]\cap X\}.
$$

Since $\pi_0\circ \phi=\pi_0\circ \psi$ and $\pi_1\circ
\phi=\pi_1\circ \psi,$ we compute that
$$
\|\pi_{\eta_1}(w)\pi_{\eta_1}\circ \psi(a)-\pi_{\eta_1}\circ
\psi(a)\pi_{\eta_1}(w)\|<\ep_1
$$
for all $a\in {\cal G}.$ It follows from \ref{almostcommut} that
there exists a continuous path of unitaries $\{W(t): t\in [0,
\eta_1]\} \subset Q$ such that $W(0)=1,$ $W(\eta_1)=w(\eta_1).$
$$
\|[W(t),\, \psi(x)]\|<\ep/4\tforal x\in {\cal F}
$$
and $t\in [0,1].$  Define $v\in C(X, Q)$ by $\pi_t(v)=W(t)$ if
$t\in [0,\eta_1]\cap X$ and $\pi_t(v)=\pi_t\circ w$ for $t\in
(\eta_1, 1]\cap X.$ Then $v\in C(X, Q)$
$$
\|{\rm ad}\, v\circ \psi(a)-\phi(a)\|<\ep/2
$$
for all $a\in {\cal F}$ and $v(0)=1.$ We can deploy the same
argument for the end point $1.$ Thus we obtain a unitary $u\in
C(X, Q)$ such that $u(0)=1=u(1)$ and
$$
\|{\rm ad}\, u\circ \psi(a)-\phi(a)\|<\ep
$$
for all $a\in {\cal F}.$

Let $F\subset Q$ be a full
matrix algebra as a unital \SCA. Let $\{e_i,j: 1\le i,j\le k\}$ be
a system of matrix unit. From the above, there is a unitary $W\in
C(X, Q)$ such that
$$
\|W^*\psi(e_{i,j})W-\phi(e_{i,j})\|<{\ep\over{8k}}\,\,\,1\le i,j\le k.
$$
One has a unitary  $v\in C(X, Q)$
such that
$$
\phi(e_{1,1})=v^*\psi(e_{1,1})v.
$$
Define
$$
W_1=\sum_{i=1}^k \psi(e_{i,1})v\phi(e_{1,i}).
$$
Then $W_1$ is a unitary such that
$$
W_1\psi(f)W_1^*=\phi(f)\tforal f\in F.
$$
For any $\ep>0,$ there exists $1>\dt>0$ such that
\beq\label{add1}
\|W_1(t)-W_1(t')\|<\ep/4
\eneq
for all $t, t'\in X$ and $|t-t'|<\dt.$


Since
$U(\phi(e_{i,i})Q\phi(e_{i,i}))=U_0(\phi(e_{i,i})Q\phi(e_{i,i})),$
it is easy to obtain a continuous path of unitaries $\{v(t): t\in
[0, \dt]\}\subset U(Q)$ such that
$$
v(\dt)=W_1(0)^*, \,\,\, v(0)=1\andeqn u(t)\pi_a\circ
\phi(f)=\pi_a\circ \phi(f)u(t)
$$
for all $f\in F$ and all $t\in [0,\dt].$ Define now $u(t)\in C(X,
Q)$ by $u(t)=v(t)$ if $t\in [0, \dt)\cap X$ and
$u(t)=W_1({t-\dt\over{1-\dt}})^*$ if $t\in [\dt, 1]\cap X.$  By
(\ref{add1}),
$$
\|{\rm ad}\, u\circ \psi(f)-\phi(f)\|<\ep\tforal f\in F.
$$
One may arrange the other end point $b$ similarly.
\end{proof}

\begin{NN}{\rm
One can avoid to using Lemma \ref{almostcommut} by modifying the
last part of the above proof to obtain \ref{Appr}.

}
\end{NN}

\section{Approximate factorizations}

\begin{df}\label{DA}
Let $X$ be a connected compact metric space. We say that $X$ has
the property (A), if, for any $\xi\in X$ and any $\dt>0,$ there
exists $\dt>\dt_1>\dt_2>\dt_3>0$ satisfying the following: there
exists a homeomorphism $\chi: B(\xi, \dt_1)\setminus \{\xi\}\to
A(\xi, \dt_1, \dt_3)$ such that
$$
\chi|_{A(\xi,\dt_1, \dt_2)}={\rm id}_{A(\xi,\dt_1, \dt_2)}.
$$

Every locally Euclidean connected compact metric space has the
property (A). All connected simplicial complex has the property
(A).
\end{df}

\begin{df}

{\rm Let $X$ be a compact metric space and let $k\ge 1$ be an
integer. Fix  $\xi_1,\xi_2,...,\xi_n\in X$ and fix positive
integers  $m_1, m_2,...,m_n$  for which $m_j|k,$ $j=1,2,...,n.$
Denote by
$$
D_{k,n}=D_k(X, \xi_1, \xi_2,...,\xi_n, m_1,m_2,...,m_n).
$$

If $X_1, X_2,...,X_n\subset X$ are disjoint compact subsets, put
$$
{\bar D}_{k,\{X_j\},n}=D_k(X, X_1, X_2,...,X_n, m_1, m_2,...,m_n),
$$
when $\{X_j: 1\le j\le n\}$ and $\{m_j: 1\le j\le n\}$ are
understood.

Let $P\in {\bar D}_{k,\{X_j\},n}$ be a projection. Then
$f(x)\cdot P(x)\in P{\bar D}_{k, n, \{X_j\}}P$ for all $f\in
C(X).$ Similarly, if $P\in D_{k,n}$ is a projection, then
$f(x)\cdot P(x)\in PD_{k,n}P.$

Denote by $C_{k, \{X_j\}, n}$ the subalgebra
$$
\{f\cdot 1_{{\bar D}_{k,n, \{X_j\}}}:  f\in C(X)\}.
$$
Then $C_{k,n, \{X_j\}}P$ is the center of $P{\bar D}_{k,n,
\{X_j\}}P.$

Similarly, set
$$
C_{k,n}=\{f\cdot 1_{D_{k,n}}: f\in C(X)\}.
$$
Then $C_{k,n}\cdot P$ is the center of $PD_{k,n}P.$

}

\end{df}

\begin{lem}\label{FL1}
Let $X$ be a connected compact metric space with property (A), let
$\xi_1, \xi_2,...,\xi_n\in X,$ let $k, m_1, m_2,...,m_n\ge 1$ be
integers with $m_j|k,$ $j=1,2,...,n.$ Let $P\in D_{k,n} $ be a
projection.

Then, for any $\ep>0,$ any finite subset ${\cal F}\subset
PD_{k,n}P\otimes Q$
there exists a unital monomorphism
$$
\psi:PD_{k,n-1}P\otimes Q\to PD_{k,n}P\otimes Q
$$
such that
$$
{\rm id}_{PD_{k,n}P\otimes Q}\approx_{\ep} \psi\circ \imath\,\,\,
{\rm on}\,\,\, {\cal F},
$$
where $$\imath: PD_{k,n}P\otimes Q\to PD_{k,n-1}P\otimes Q$$
 is the embedding.

Moreover, $\psi$ maps $(C_{k,n-1}P)\otimes 1_Q$ into
$(C_{k,n}P)\otimes 1_Q.$

\end{lem}

\begin{proof}

We may  assume that $\|f\|\le 1$ for all $f\in {\cal F}.$
Fix $\ep>0$ and a finite subset ${\cal F}\subset PD_{k,n}P\otimes
Q.$ We may assume that $P\in {\cal F}.$ We may  assume that
$\|f\|\le 1$ for all $f\in {\cal F}$ and ${\cal F}\subset
PD_{k,n}P\otimes M_{m!},$ where $M_{m!}$ is a unital \SCA\, of
$Q.$
There is $\dt>0$ such that
$$
\|f(x)-f(x')\|<\ep/16\tforal f\in {\cal F}
$$
whenever ${\rm dist}(x,x')\le \dt.$

Without loss of generality, we may also assume that ${\rm
dist}(\xi_j, \xi_n)>\dt,$ $j=1,2,...,n-1.$
There are $\dt>\dt_1>\dt_2>\dt_3>0$ such that there exists a
homeomorphism $\chi: B(\xi_n,\dt_1)\setminus \{\xi_n\}\to A(\xi_n,
\dt_1,\dt_3)$ such that $\chi|_{A(\xi_n, \dt_1, \dt_2)}={\rm
id}|_{A(\xi_n, \dt_1, \dt_2)}.$


For each $f\in C(X, M_k)\otimes Q$ define ${\tilde f}(x)=f(x)$ if
${\rm dist}(x, \xi_n)>\dt_2,$ ${\tilde f}(x)=f(\xi_n)$ if ${\rm
dist}(x, \xi_n)\le \dt_3$ and ${\tilde f}(x)=f(\chi^{-1}(x))$ if
$x\in A(\xi_n, \dt_2, \dt_3).$  Note that, for each $f\in {\cal
F},$
\beq\label{FL1-1}
 \|{\tilde f}-f\|<\ep/16.
\eneq
 Moreover ${\tilde P}$ is also a projection.
Note that, for each $x\in X, $ $\pi_x({\tilde P}D_{k, m}{\tilde
P}\otimes Q)\cong Q$ for $m\ge 1.$

There are $\dt_3>\dt_4>\dt_5>0$ such that there is a homeomorphism
$\chi_1: B(\xi_n, \dt_3)\setminus \{\xi_n\}\to A(\xi_n, \dt_3,
\dt_5)$ such that  $(\chi_1)|_{A(\xi_n,\dt_3, \dt_4)}={\rm
id}|_{A(\xi_n,\dt_3, \dt_4)}.$

Let $0<\dt_6<\dt_5.$
 There is a unital isomorphism  $\lambda: \pi_{\xi_n}({\tilde P}D_{k,
n-1}{\tilde P}\otimes Q)\to \pi_{\xi_n}({\tilde P}D_{k, n}{\tilde
P}\otimes Q).$  Note that $\pi_{\xi_n}({\tilde P}D_{k,n}{\tilde
P}\otimes Q)$ is a unital \SCA. In particular.
\beq\label{FL1-1+}
(\lambda|_{\pi_{\xi_n}({\tilde P}D_{k,n}{\tilde P}\otimes
Q)})_{*0}={\rm id}_{K_0(\pi_{\xi_n}({\tilde P}D_{k,n}{\tilde
P}\otimes Q))}.
\eneq
Therefore, there is a unitary $U_0\in \pi_{\xi_n}({\tilde P}D_{k,
n}{\tilde P}\otimes Q)$ such that
$$
{\rm ad}\, U_0\circ \lambda(a)=a
$$
for all $a\in \pi_{\xi_n}({\tilde P}D_{k,n}{\tilde P}\otimes
M_{m!}).$ To simplify the notation, we may assume that
\beq\label{FL1-++}
\lambda(a)=a\tforal a\in \pi_{\xi_n}({\tilde P}D_{k,n}{\tilde
P}\otimes M_{m!}).
\eneq

On the other hand, by embedding $\pi_{\xi_n}({\tilde
P}D_{k,n}{\tilde P}\otimes Q)$ into $\pi_{\xi_n}({\tilde
P}D_{k,n-1}{\tilde P}\otimes Q)$ unitally, we may view $\lambda$
maps $\pi_{\xi_n}({\tilde P}D_{k,n-1}{\tilde P}\otimes Q)$ into
$\pi_{\xi_n}({\tilde P}D_{k,n-1}{\tilde P}\otimes Q).$ Then
$$
(\imath\circ \lambda)_{*0}={\rm id}_{K_0(\pi_{\xi_n}({\tilde
P}D_{k,n-1}{\tilde P}\otimes Q))}.
$$
It then follows from \ref{asy}, there is a continuous path of
unitaries
$$
\{u(t): t\in [0, \dt_6)\}\subset \pi_{\xi_n}({\tilde P}D_{k,
n-1}{\tilde P}\otimes Q)
$$ such that $u(0)=1$ and
$$
\lim_{t\to \dt_6}{\rm ad}\, u(t)\circ \lambda(f)=f
$$
for all $f\in \pi_{\xi_n}({\tilde P}D_{k, n-1}{\tilde P}\otimes Q).$

Define a projection $P_1\in C(X, M_k)\otimes Q$ as follows:
$P_1(x)={\tilde P}(x)\otimes 1_Q$ for all ${\rm dist}(x, \xi_n)\ge
\dt_4,$ $P_1(x)={\tilde P}(\chi_1^{-1}(x))\otimes 1_Q,$ if $x\in
A(\xi_n, \dt_4, \dt_5)$ and
$$
P_1(x)=u^*(t)({\tilde P}(\xi_n)\otimes 1_Q)u(t)=P(\xi_n)\otimes
1_Q,
$$
if ${\rm dist}(x,\xi_n)=t$ for $0\le t\le \dt_6,$ and
$P_1(x)={\tilde P}(\xi_n)\otimes 1_Q$ if $\dt_5>{\rm
dist}(x,\xi_n)\ge \dt_6.$ Thus $P_1={\tilde P}\otimes 1_Q.$

 Define $\Phi: {\tilde P}D_{k,n-1}{\tilde P}\otimes Q=P_1(D_{k,n-1}\otimes Q)P_1\to P_1(D_{k, n}\otimes Q)P_1$ as follows:
$\Phi(f)(x)=f(x),$ if ${\rm dist}(x, \xi_n)\ge \dt_4,$
$\Phi(f)(x)=f(\chi_1^{-1}(x)),$ if $x\in A(\xi_n, \dt_4, \dt_5),$
$\Phi(f)(x)=f(\xi_n),$ if $\dt_6\le {\rm dist}(x, \xi_n)<\dt_5$
and
$$
\Phi(f)(x)=u^*(t)\lambda(f)u(t)
$$
if ${\rm dist}(x, \xi_n)=t$ for $0\le t< \dt_6.$ It follows that
\beq\label{FL1-2}
\|\Phi({\tilde f})(x)-{\tilde f}(x)\|<\ep/16
\eneq
for all $f\in {\cal F}$ and for those $x\in X$ with ${\rm dist}(x,
\xi_n)\ge\dt_6.$

Define a continuous map $\Gamma: B(\xi_n, \dt_5)\to [0,\dt_5]$ by
$\Gamma(x)={\rm dist}(x, \xi_n)$ for $x\in B(\xi_n, \dt_5).$ Put
$Y=\overline{\Gamma(B(\xi_n, \dt_6))}$ and
$Y_1=\overline{\Gamma(B(\xi_n, \dt_5)}.$
Define a projection $P_0\in C(Y, \pi_{\xi_n}(P_1(D_{k,n-1}\otimes
Q)P_1))$ by
$$
P_0(t)=u^*(t)({\tilde P}(\xi_n)\otimes 1_Q)u(t).
$$
We note that $P_0(t)={\tilde P}(\xi_n)\otimes 1_Q$ for all $t\in
Y.$

Define $\psi_0: \pi_{\xi_n}({\tilde P}D_{k, n}{\tilde P}\otimes
Q)\to C(Y, P_1(\xi_n)M_k(Q)P_1(\xi_n))\, (\cong C(Y, Q))$ by
$$
\psi_0(a)(t)=a
$$
for all $a\in \pi_{\xi_n}({\tilde P}D_{k, n}{\tilde P}\otimes Q).$
Define $\psi_1: \pi_{\xi_n}({\tilde P}D_{k, n}{\tilde P}\otimes
Q)\to C(Y, P_1(\xi_n)M_k(Q)P_1(\xi_n))$ by
$$
\psi_1(a)(t)=u^*(t)\lambda(a)u(t)
$$
for $a\in \pi_{\xi_n}({\tilde P}D_{k, n}{\tilde P}\otimes Q),$
$t\in Y\setminus \{1\}$ and $\psi(a)(\dt_6)=a$ for all $a\in
\pi_{\xi_n}({\tilde P}D_{k,n}{\tilde P}\otimes Q).$ It follows
from \ref{Appr} that there exists a unitary
$$
U_1\in C(Y, P_1(\xi_n)M_k(Q)P_1(\xi_n))
$$
such that $U_1(0)=U_1(\dt_6)=P_1(\xi_n)$ and
\beq\label{FL1-2+}
(U_1)^*\psi_1U_1\approx_{\ep/16} \psi_0\,\,\,{\rm
on}\,\,\,\{f(\xi_n): f\in {\cal F}\}.
\eneq
Put $U_2=U_1+(1_{M_k(Q)}-P_1(\xi_n)).$ Let $U_3\in C(X, M_k(Q))$
be defined by $ U_3(x)=1 $ if ${\rm dist}(x, \xi_n)\ge \dt_6$ and
$U_3(x)=U_2(t)$ if ${\rm dist}(x, \xi_n)=t$ for $t\in [0,\dt_6).$

We have, by (\ref{FL1-2}),
\beq\label{FL1-3}
\|U_3^*(x){\tilde \Phi}({\tilde f})U_3(t)-{\tilde f}(x)\|<\ep/16
\eneq
for all $f\in {\cal F}$ and $x\in X$ with ${\rm dist}(x, \xi_n)\ge
\dt_6,$ and, by (\ref{FL1-2+}),
\beq\label{FL1-4}
\|U_3^*(x)\Phi({\tilde f})U_3(x)-{\tilde f}(x)\|<\ep/16
\eneq
for all $f\in {\cal F}$ and $x\in X$ with ${\rm dist}(x,
\xi_n)<\dt_6.$

 Note that
 $$
 \|P-{\tilde P}\|<\ep/16.
 $$
There is a unitary  $U_4'\in C(X, M_k)$ such that
$$
\|U_4'-1\|<\sqrt{2}\ep/16
$$
and
$$
(U_4')^*{\tilde P}U_4'=P.
$$
(see Lemma 6.2.1 of \cite{Mb} for example). Then $U_4'(x)=P(x)$ if
${\rm dist}(x,\xi_n)\ge \dt_3$ and $U_4'(\xi_n)=P(\xi_n).$ Define
$U_4=(1_{C(X, M_k)\otimes Q}-1_{C(X,M_k)})+U_4''\in
C(X,M_k)\otimes 1_Q.$

Define $\psi: PD_{k, n-1}P\otimes Q\to PD_{k, n}P\otimes Q$ by
$$
\psi(f)={\rm ad}\, U_4 U_3\circ  \Phi(f)\tforal f\in
PD_{k,n-1}P\otimes Q.
$$

Now we estimate (by applying (\ref{FL1-3}) and \ref{FL1-4})) that
\beq
 \|f-\psi(f)\| &\le & \|f-{\tilde f}\|+\|{\tilde f}-\psi({\tilde
 f})\|\\
 &<& \ep/16 +\sqrt{2}\ep/16+\|{\tilde f}-U_3^*\Phi({\tilde
 f})U_3\|\\
 &<& {(1+\sqrt{2})\ep\over{16}}+2\ep/16<\ep
\eneq for all $f\in {\cal F}.$
Note that $\psi$ is a unital monomorphism.

This proves the first part of the lemma. 
The last part follows from the 
construction that $\psi$ maps $(C_{k, n-1}P)\otimes 1_Q$ into
$(C_{k,n}P)\otimes 1_Q.$

\end{proof}

\begin{lem}\label{FL2}
Let $X$ be a connected compact metric space with property (A), let
$\xi_1, \xi_2,...,\xi_n\in X,$ let $k, m_1, m_2,...,m_n\ge 1$ be
integers with $m_j|k,$ $j=1,2,...,n.$ Let $P\in D_k(X,
\xi_1,\xi_2,...,\xi_n, m_1, m_2,...,m_n)$ be a projection.
Then, for any $\ep>0,$ any finite subset ${\cal F}\subset PC(X,
M_k)P\otimes Q$
there exists a unital monomorphism
$$
\psi:PC(X, M_k)P\otimes Q\to PD_{k,n}P\otimes Q
$$
such that
$$
{\rm id}_{PD_{k,n}P\otimes Q}\approx_{\ep} \psi\circ \imath\,\,\,
{\rm on}\,\,\, {\cal F},
$$
where $$\imath: PD_{k,n}P\otimes Q\to PC(X, M_k)P\otimes Q$$
 is the embedding.

Moreover, $\psi$ maps $(C(X)\cdot P)\otimes 1_Q$ into
$(C_{k,n}P)\otimes 1_Q.$

\end{lem}

\begin{proof}
We prove this by induction on $n.$ If $n=1,$ then the lemma
follows from \ref{FL1} immediately.

Suppose that the lemma holds for integers $1\le n\le N.$
Denote by $D_{k,l}=D_k(X, \xi_1, \xi_2,....,\xi_{l}, m_1,
m_2,...,m_{l}),$ $l=1,2,...,.$

Let $\ep>0,$ let ${\cal F}\subset PC(X, M_k)P\otimes Q$ and let
$F\subset Q$ be a unital finite dimensional \SCA.  By the
inductive assumption, there exists a unital monomorphism $\psi_1:
PC(X, M_k)P\otimes Q\to PD_{k, N}P\otimes Q$ such that
\beq\label{FL2-1}
\|\psi_1(f)-f\|<\ep/2
\eneq
for all $f\in {\cal F}.$ Moreover there exists a unital finite
dimensional \SCA\, $F_0\subset Q$ such that
$$
\psi_1(PC(X, M_k)P\otimes F)\subset C_F'.
$$

By Lemma \ref{FL1}, there exists a unital monomorphism $\psi_2:
PD_{k,N}P\otimes Q\to PD_{k, N+1}P\otimes Q$ such that
\beq\label{FL2-2}
\|\psi_2(f)-f\|<\ep/2
\eneq
for all $f\in W_1^*\psi_1({\cal F})W_1$ and a unital finite
dimensional \SCA\, $F_1\in Q$ such that
$$
\psi_2(PD_{k,N}P\otimes F_0)\subset C_{F_0}.
$$

Define $\psi: PC(X, M_k)P\otimes Q\to PD_{k, N+1}P\otimes Q$ by
$\psi= {\rm ad}\, W_1^*\circ \psi_2\circ {\rm ad}\, W_1\circ
\psi_1.$ Then
\beq
\|\psi(f)-f\|&=&\| W_1\psi_2(W_1^*\psi_1(f)W_1)W_1^*-f\|\\
&\le
&\|W_1\psi_2(W_1^*\psi_1(f)W_1)W_1^*-\psi_1(f)\|+\|\psi_1(f)-f\|\\
&\le &\|\psi_2(W_1^*\psi_1(f)W_1)-W_1^*\psi_1(f)W_1\|
+\ep/2<\ep/2+\ep/2=\ep
\eneq
for all $f\in {\cal F}.$

Moreover, by \ref{FL1}, $\psi$ maps $(C(X)\cdot P)\otimes 1_Q$
into $(D_{k,n}P)\otimes 1_Q.$
This completes the induction.

\end{proof}

\begin{lem}\label{FL3}
Let $X$ be a  connected compact metric space with property (A),
let $X_1, X_2,...,X_n\subset X$ be compact subsets,  let $k, m_1,
m_2,...,m_n\ge 1$ be integers with $m_j|k,$ $j=1,2,...,n.$ Let
$P\in {\bar D}_{k, \{X_j\},n}= {\bar D}_k(X, X_1,X_2,...,X_n, m_1,
m_2,...,m_n)$ be a projection.

Then, for any $\ep>0,$ any finite subset ${\cal F}\subset PC(X,
M_k)P\otimes Q$
there exists a unital monomorphism
$$
\psi:PC(X, M_k)P\otimes Q\to P{\bar D}_{k, \{X_j\},n}P\otimes Q
$$
such that
$$
{\rm id}_{P{\bar D}_{k,\{X_j\},n}P\otimes Q}\approx_{\ep}
\psi\circ \imath\,\,\, {\rm on}\,\,\, {\cal F},
$$
where $$\imath: P{\bar D}_{k,\{X_j\},n}P\otimes Q\to PC(X,
M_k)P\otimes Q$$
 is the embedding.

Moreover, $\psi$ maps $(C(X)\cdot P)\otimes 1_Q$ into $(C_{k,
\{X_j\},n}P)\otimes 1_Q.$

Furthermore, if $F\subset Q$ is a unital finite dimensional \SCA,
we may choose $\psi$ so that there exists another unital finite
dimensional \SCA\, $F_1\subset Q$ and a unital \SCA\, $$C_F\subset
PD_{k,n}P\otimes Q$$ such that
$$
\psi(PC(X, M_k))P\otimes F)\subset C_F
$$
and there is a unitary $W\in M_2(P{\bar D}_{k, \{X_j\}, n}P\otimes
Q)$ such that
$$
W^*C_FW=P_1M_2(P{\bar D}_{k, \{X_j\},n} P\otimes F_1)P_1
$$
for some projection $P_1\in M_2(P{\bar D}_{k, \{X_j\}, n}P\otimes
F_1).$
\end{lem}

\begin{proof}

Let $\{\xi_i: i=1,2,...\}$ be a dense sequence of $X_n.$ Put
$$
D_{k,\{X_j\}_{j=1}^{n-1}, \{\xi_i\}, l}={\bar D}_k(X, X_1,
X_2,...,X_{n-1}, \xi_1,\xi_2,...,\xi_l, m_1,
m_2,...,m_{l-1},\overbrace{m_n, m_n,...,m_n}^l).
$$
Note that $PD_{k, \{X_j\}_{j=1}^{n-1},\{\xi_i\}, l+1}P\subset
PD_{k,\{X_j\}_{j=1}^{n-1}, \{\xi_i\},l}P$ and
$$
\bigcap_{l=1}^{\infty} PD_{k,\{X_j\}_{j=1}^{n-1},\{\xi_i\},
l}P=P{\bar D}_{k, \{X_j\}, n}P.
$$

Let $\ep>0$ and let a finite subset  ${\cal F}\subset P{\bar
D}_{k, \{X_j\}_{j=1}^{n-1}, n-1}P\otimes Q$ be  given.

Let $\{{\cal F}_{0,m}\}$ be an increasing sequence of finite
subsets of $P{\bar D}_{k, \{X_j\}_{j=1}^{n-1}, n-1}P\otimes Q$
whose union is dense in $P{\bar D}_{k, \{X_j\}_{j=1}^{n-1},
n-1}P\otimes Q.$ Let $\{{\cal F}_{\infty,m}\}$ be an increasing
sequence of finite subsets of $P{\bar D}_{k,\{X_j\}_{j=1}^n,
n}P\otimes Q$ whose union is dense in $P{\bar D}_{k,
\{X_j\}_{j=1}^{n},n}P\otimes Q.$ Let $\{{\cal F}_{l,m}\}$ be an an
increasing sequence of finite subsets of
$PD_{k,\{X_j\}_{j=1}^{n-1}, \{\xi_i\},l}P\otimes Q$ whose union is
dense in $PD_{k,\{X_j\}_{j=1}^{n-1}, \{\xi_i\}, l}P\otimes Q.$  By replacing 
$F_{j,m}$ by $\cup_{s\ge j} F_{s,m},$ we
may assume that ${\cal F}_{j+1, m}\subset {\cal F}_{j, m},$
$j=0,1,2,...,$ and ${\cal F}_{\infty, m}\subset {\cal F}_{j,m},$
$j=1,2,....$ Let $\{\ep_n\}$ be a sequence of decreasing positive
numbers such that $\sum_{n=1}^{\infty}\ep_n<\ep/2.$ We may also
assume that ${\cal F}={\cal F}_{0,1}.$


By applying \ref{FL2}, one obtains, for each $l,$  a unital
monomorphism $\psi_l: PD_{k,\{X_j\}_{j=1}^{n-1}, \{\xi_i\},
l-1}P\otimes Q\to PD_{k,\{X_j\}_{j=1}^{n-1}, \{\xi_i\},l}P\otimes
Q$ which maps $C_{k,\{X_j\}_{j=1}^{n-1}, \{\xi_i\}, l-1}\cdot P$
into $C_{k,\{X_j\}_{j=1}^{n-1}, \{\xi_i\}, l}P$ such that
\beq\label{FL3-1}
{\rm id}_{PD_{k,\{X_j\}_{j=1}^{n-1}, \{\xi_i\}, l}P\otimes
Q}\approx_{\ep_l/2} \psi_n\circ \imath\,\,\,{\rm on}\,\,\, {\cal
G}_{l-1, l},
\eneq
where ${\cal G}_{0,1}={\cal F}_{0,1}={\cal F}$ and ${\cal
G}_{l-1,l}={\cal F}_{l-1, l}\cup \psi_{l-1}({\cal G}_{l-1, l}),$
$l=2,3,....$

Define $A=\lim_{l\to\infty}(PD_{k,\{X_j\}_{j=1}^{n-1}, \{\xi_i\},
l-1}P\otimes Q, \psi_l).$ We claim that $A\cong P{\bar D}_{k,
\{X_j\},n}P\otimes Q.$

Put $B_l=PD_{k,\{X_j\}_{j=1}^{n-1}, \{\xi_i\}, l-1}P\otimes Q,$
$l=1,2,...$ and put $C=C(X, M_k)\otimes Q.$ Note that $B_1=P{\bar
D}_{k, \{X_j\}_{j=1}^{n-1}, n-1}P\otimes Q.$

 Consider the following diagram:
$$
\begin{array}{ccccccc}
B_1 &\stackrel{\psi_1}{\to} & B_2 &\stackrel{\psi_2}{\to}& B_3
&\stackrel{\psi_2}{\to}\cdots\to A \\
\downarrow_{\imath} && \downarrow_{\imath} & & \downarrow_{\imath} &&\\
 C &\stackrel{{\rm id}}{\to} & C &\stackrel{{\rm id}}{\to}& C
&\stackrel{{\rm id}}{\to}\cdots \to C.
\end{array}
$$
It follows from (\ref{FL3-1}) that the diagram is one-sided
approximately intertwining. Therefore, by an argument of Elliott
(see Theorem 1.10.4 of \cite{Lnbk}, for example) there exists a
\hm\, $\Phi: A\to C=C(X, M_k)\otimes Q$ such that
\beq\label{FL3-2}
\Phi\circ \psi_{l, \infty}\approx_{\sum_{j=l}^{\infty}\ep_l}
\imath\,\,\,{\rm on}\,\,\, {\cal G}_{l-1, l}\andeqn\\\label{FL3-3}
\Phi\circ \psi_{l, \infty}(b)=\lim_{s\to\infty} \imath\circ
\psi_{l,s}(b)
\eneq
for all $b\in B_{l-1}.$ Combining (\ref{FL3-1}) and (\ref{FL3-3}),
we conclude that
\beq\label{FL3-4}
{\rm dist}(\Phi\circ \psi_{l, \infty}(b), C_l)=0
\eneq
for all $b\in B_l.$  Since $C_m\subset C_l,$ if $m\ge l,$ we
conclude that
\beq\label{FL3-5}
{\rm dist}(\Phi\circ \psi_{m, \infty}(b), C_l)=0
\eneq
for all $b\in C_m,$ $m=l, l+1,....$ It follows that
\beq\label{FL3-6}
{\rm dist}(\Phi(A), C_l)=0\tforal l.
\eneq
Therefore
\beq\label{FL3-7}
\Phi(A)\subset \cap_{l=1}^{\infty}C_l=PD_{k, \{X_j\},n}P\otimes Q.
\eneq

Put $D=PD_{k, \{X_j\},n}P\otimes Q.$ Then $\imath: D\to B_l$ for
each $l.$ Thus, as above,  we obtain the following one-sided
approximately intertwining:
$$
\begin{array}{ccccccc}
D &\stackrel{{\rm id}}{\to} & D &\stackrel{{\rm id}}{\to}& D
&\stackrel{{\rm id}}{\to}\cdots \to D \\
\downarrow_{\imath} && \downarrow_{\imath} & & \downarrow_{\imath} &&\\
 B_1 &\stackrel{\psi_1}{\to} & B_2 &\stackrel{\psi_2}{\to}& B_3
&\stackrel{\psi_3}{\to}\cdots \to A.
\end{array}
$$
From this we obtain a \hm\, $J: D\to A$ such that
\beq\label{FL3-8}
J(d)=\lim_{l\to\infty}\psi_{l,\infty}\circ \imath(d)
\eneq
for all $d\in D.$ We then compute (by \ref{FL3-3}) and
(\ref{FL3-1}))  that
\beq\label{FL3-9}
\Phi\circ J={\rm id}_D.
\eneq
This, in particular, implies that $\Phi$ maps $A$ onto $D.$ We
then conclude that $\Phi$  gives an isomorphism from $A$ to $D$
and $\Phi^{-1}=J.$

Put $\Psi=\Phi\circ \psi_{1, \infty}: PD_{k, \{X_j\}_{j=1}^{n-1},
n-1}P\otimes Q\to D.$ Then, by (\ref{FL3-2}), we have
\beq\label{FL3-10}
{\rm id}_{D}\approx_{\ep/2} \Psi\circ \imath\,\,\,{\rm on}\,\,\,
{\cal F}={\cal F}_{0,1}.
\eneq
Moreover, $\Psi$ is a unital monomorphism. Furthermore,  we also note
that $\Psi$ maps $C_{k, \{X_j\}_{j=1}^{n-1},n-1}P\otimes 1_Q$ into
$C_{k, \{X_j\}, n}P\otimes 1_Q.$

We now use the induction.  The same argument used in the proof of
\ref{FL2} proves that there exists a unital monomorphism $\psi:
PC(X, M_k)P\otimes Q\to  PD_{k, \{X_j\}, n}P\otimes Q$ such that
\beq\label{FL3-11}
{\rm id}_D\approx_{\ep} \psi\circ \imath\,\,\,{\rm on}\,\,\, {\cal
F}
\eneq
and $\psi$ maps $(C(X, M_k)\cdot P)\otimes 1_Q)$ into $(C_{k,
\{X_j\},n}\cdot P)\otimes 1_Q).$ This proves the first two parts
of the statement.

To prove the last part of the statement, let $F_0\subset Q$ be a
unital finite dimensional \SCA. We may assume, without loss of
generality,  that $F_0\subset M_{m!}.$ We may also assume that
$m>2({\rm dim}(X)+1)$ and $2({\rm rank}(P)\times m!-{\rm
dim}(X))\ge {\rm rank}(P)\times m!.$ Put $r={\rm rank}(P)\times
m!-{\rm dim}(X).$ There  is a system of matrix units
$\{e_{i,j}\}_{i,j=1}^r\subset PC(X, M_k)P\otimes M_{m!}$ such that
$e_{i,i}$ is a rank one trivial projections and
$$
P\otimes 1_{M_{m!}}=q_0\oplus \sum_{i=1}^re_{i,i}.
$$
where $q_0$ is a projection of rank ${\rm dim}(X).$  Let $E$ be
the \SCA\, generated by $(C(X)\cdot P\otimes 1_{M_{m!}})\cdot
e_{1,1}$ and $\{e_{i,j}: 1\le i,j\le r\}.$ There is a projection
$q_{00}\in E$ with rank ${\rm dim}(X)$ and a unitary $w_0\in PC(X,
M_k)P\otimes M_{m!}$ such that
$$
w_0^*q_{00}w_0=q_0.
$$
It follows that there exists a projection $P_1\in  M_2(E)$ and a
unitary $W_0\in M_2(PC(X, M_k)P\otimes M_{m!})$ such that
\beq\label{FL3-12}
W_0^*P_1M_2(E)P_1W_0=PC(X, M_k)P\otimes M_{m!},
\eneq

For any $\dt>0,$  there is an integer $N\ge m$ such that
$$
{\rm dist}(\psi(e_{i,j}), D\otimes M_{N!})<\dt
$$
for some integer $N\ge m,$ $1\le i,j\le r.$ It follows from Lemma
2.5.10 of \cite{Lnbk} that, by choosing
small $\dt,$  there exists a unitary $W_1\in D\otimes M_{N!}$ such
that
$$
\|W_1-1\|<\ep/2\andeqn W_1^*\psi(e_{i,j})W_1\subset D\otimes
M_{N!},\,\,\,i,j=1,2,...,r.
$$

Since $\psi(C(X)\cdot P\otimes 1_Q)$ is in the center of $D\otimes
Q,$
\beq
\psi(fW_1^*e_{i,j}W_1)&=&\psi(f)\psi(W_1)^*\psi(e_{i,j})\psi(W_1)\\
&=&\psi(W_1)^*\psi(f)\psi(e_{i,j})\psi(W_1)=\psi(W_1)^*\psi(fe_{i,j})\psi(W_1)
\eneq
for all $f\in C(X)\cdot P\otimes 1_Q$ and $1\le i,j\le r.$  It
follows that
$$
\psi(W_1)^*\psi(E)\psi(W_1)\subset  PD_{k, \{X_j\},n}P\otimes
M_{N!}.
$$
Put $P_2=(\psi\otimes {\rm id}_{M_2})(P_1).$ Then $P_2\in
M_2(PD_{k, \{X_j\}, n}P\otimes M_{N!}).$
Define $W_3=\begin{pmatrix}\psi(W_1) & 0\\
                                                       0 & \psi(W_1)\end{pmatrix}.$
By (\ref{FL3-12}),
\beq\label{FL3-13}
\psi(PC(X, M_k)P\otimes M_{m!})&=&(\psi\otimes {\rm id}_{M_2})
(PC(X, M_k)P\otimes M_{m!})\\
 &&\hspace{-1.4in}=(\psi\otimes{\rm id}_{M_2})(W_0^*P_1M_2(E)P_1W_0)\\
 &&\hspace{-1.4in}=
(\psi\otimes {\rm id}_{M_2})(W_0)^*(P_2M_2(\psi(E))P_2)
(\psi\otimes{\rm id}_{M_2})(W_0)\\
&&\hspace{-1.4in}=
(\psi\otimes {\rm
id}_{M_2})(W_0)^*(W_3(W_3^*P_2M_2(\psi(E))P_2)W_3)W_3^*)
(\psi\otimes{\rm id}_{M_2})(W_0)\\
 &&\hspace{-1.4in}\subset
\psi\otimes {\rm id}_{M_2}(W_0)^*(W_3(P_2M_2(PD_{k,
\{X_j\},n}P\otimes M_{N!})P_2)W_3^* )(\psi\otimes{\rm
id}_{M_2})(W_0)
\eneq
Put $W=W_3^*(\psi\otimes {\rm id}_{M_2}(W_0)).$


\end{proof}

\section{The main results}

\begin{thm}\label{MT3}
Let $A=\lim_{n\to\infty}(A_n, \phi_n)$ be a unital simple
inductive limit of generalized dimension drop algebras. Then
$TR(A\otimes Q)\le 1.$
\end{thm}

\begin{proof}
Write $Q=\overline{\cup_{n=1}^{\infty}F_n},$ where each $F_n$ is a
unital simple finite dimensional \SCA s and $F_n\subset F_{n+1},$
$n=1,2,....$
Let
$\Phi_n=\phi_n\otimes {\rm id}_Q: A_n\otimes Q\to A_{n+1}\otimes
Q,$ $n=1,2,....$  Write
$$
A_n=\bigoplus_{j=1}^{m(n)}P_{n, j}D(X_{n,j},r(n,j))P_{n,j},
$$
where $X_{n,j}$ is a connected compact metric space with finite
dimension and with property (A), $r(n,j)\ge 1$ is an integer,
$$
D(X_{n,j}, r(n,j))= {\bar D}_{r(n,j)}(X_{n,j}, r(n,j), Y_{n,j,1},
,...,Y_{n,j,d(n,j)}, m_{n,j,1},...,m_{n,j,d(n,j)}),
$$
$Y_{n,j,1}, Y_{n,j,2},...,Y_{n,j,d(n,j)}$ are disjoint compact
subsets of $X_{n,j}$ and $P_{n,j}\in D(X_{n,j}, r(n,j))$ is a
projection, $n=1,2,....$ Define
$$
B_n=\bigoplus_{j=1}^{m(n)} P_{n,j}C(X_{n,j}, M_{r(n,j)})P_{n,j},
$$
$n=1,2,....$ Denote by $\imath_n: A_n\to B_n$ the embedding.
Denote
$$
d_n=\max\{{\rm dim }(X_{n,j}: 1\le j\le m(n)\}.
$$
Let
$\{{\cal F}_n\}$ be an increasing sequence of finite subsets of
$A\otimes Q$ for which ${\cal F}_n\subset
\Phi_{n,\infty}(A_n\otimes Q),$ $n=1,2,...,$ and
$\cup_{n=1}^{\infty}{\cal F}_n$ is dense in $A\otimes Q.$ Let
${\cal G}_n\subset A_n\otimes Q$ be a finite subset such that
$\Phi_{n,\infty}({\cal G}_n)={\cal F}_n,$ $n=1,2,....$ Without
loss of generality, we may assume that ${\cal G}_n\subset
A_n\otimes F_n,$ $n=1,2,....$
Let $\{\ep_n\}$ be a sequence of decreasing positive numbers such
that $\sum_{n=1}^{\infty} \ep_n<1/2.$

It follows from \ref{FL3} that there exists  a unital monomorphism
$\Psi_1:B_1\otimes Q\to A_1\otimes Q$ such that
\beq\label{MT-1}
\Psi_1\circ \imath_1\approx_{\ep_1} {\rm id}_{A_1\otimes
Q}\,\,\,{\rm on}\,\,\, {\cal G}_1.
\eneq
 Moreover, by \ref{FL3}, there exists a unital finite dimensional
\SCA\, $E_{1,1}\subset Q,$ a unital \SCA\, $C_{F_1}\subset
A_1\otimes Q,$ a projection $P_1\in M_2(A_1\otimes E_{1,1})$ and
a unitary $W_1\in M_2(A_1\otimes Q)$ such that
$$
\Psi_1(B_1\otimes F_1)\subset C_{F_1}\andeqn W_1^*C_{F_1}W_1=
P_1M_2(A_1\otimes E_{1,1})P_1.
$$
Put $U_1=1_{B_1}$ and $C_1=B_1\otimes F_1.$
 Define $\theta_1=\Phi_1\circ \Psi_1: B_1\otimes Q\to
A_2\otimes Q.$ Suppose that $\{{\cal S}_{1,k}\}$ is an increasing
sequence of finite subsets of $B_1\otimes Q$ such that
$\cup_{k=1}^{\infty} {\cal S}_{1,k}$ is dense in $B_1\otimes Q.$
We may assume that
$$
{\cal S}_{1,k}\subset \bigoplus_{j=1}^{m(1)} P_{1,j}C(X_{1,j},
M_{r(1,j)})P_{1,j} \otimes F_k.
$$

Define ${\cal S}_1={\cal S}_{1,1}\cup \imath_1({\cal G}_1)$ and
${\cal G}_2'={\cal G}_2\cup \theta_1({\cal S}_1).$ To simplify the
notation, we may assume, without loss of generality, that there
exists a finite subset ${\cal G}_2''\subset A_2\otimes F_2$ such
that
$$
{\cal G}_2'\subset_{\ep_2/2} {\cal G}_2''.
$$
By applying \ref{FL3},
there exists a unital monomorphism $\Psi_2:B_2\otimes Q\to
A_2\otimes Q$ such that
\beq\label{MT-2}
\Psi_2\circ \imath_2\approx_{\ep_2} {\rm id}_{B_2}\,\,\,{\rm
on}\,\,\, {\cal G}_2'\cup{\cal G}_2''.
\eneq
Moreover, by \ref{FL3}, there exists a unital finite dimensional
\SCA\, $E_{2,1}\subset Q,$ a unital \SCA\, $C_{F_2}\subset
A_2\otimes Q,$ a projection $P_2\in M_2(A_2\otimes E_{2,1})$  and
a unitary $W_2\in M_2(P_2(A_2\otimes Q)P_2)$ such that
\beq
\Psi_2(B_2\otimes E_{1,1})\subset C_{F_2}\andeqn\\
W_2^*C_{F_2}W_2= P_2M_2(A_2\otimes E_{2,1})P_2.
\eneq

Define $\psi_1=\imath_2\circ\Phi_1\circ \Psi_1: B_1\otimes Q\to
B_2\otimes Q.$ Define $\theta_2=\Phi_2\circ \Psi_2.$  We have
\beq\label{MT-3} \imath_2\circ \theta_1=\psi_1. \eneq Also, by
(\ref{MT-1}) and (\ref{MT-2}), we have
\beq\label{MT-4}
\theta_1\circ \imath_1&\approx_{\ep_1}& \Phi_1\,\,\,{\rm
on}\,\,\,{\cal G}_1\andeqn\\
\theta_2\circ \imath_2&\approx_{\ep_2} & \Phi_2\,\,\,{\rm
on}\,\,\, {\cal G}_2'\cup {\cal G}_2''.
\eneq

 Let $\{{\cal S}_{2,k}\}$ be an increasing sequence of finite
subsets of $B_2\otimes Q$ such that its union is dense in
$B_2\otimes Q.$ We may assume that
\beq\label{MT-4+1}
\psi_1(S_{1,2}), \imath_2({\cal G}_2'\cup{\cal
G}_2'')\subset_{\ep_3/2} {\cal S}_{2,2}.
\eneq
We may also assume that $(\Phi_1\otimes {\rm
id}_{M_2})(W_1^*){\cal S}_{2,k}(\Phi_1\otimes {\rm
id}_{M_2})(W_1)\subset (\Phi_1\otimes {\rm
id}_{M_2})(P_1)M_2(B_2\otimes E_{2,k})(\Phi_1\otimes {\rm
id}_{M_2})(P_1),$ where $\{E_{2,k}\}$ is an increasing sequence of
finite dimensional \SCA s of $Q$ whose union is dense in $Q.$ We
assume that $E_{1,1}\subset E_{2,2}$ and $E_{2,2}$ is simple and
$$
{d_2^3\over{{\rm rank}(E_{2,2})}}<1/2^2.
$$
Define $C_2=W_1(\Phi_1\otimes {\rm id}_{M_2})(P_1)M_2(B_2\otimes
E_{2,2})(\Phi_1\otimes {\rm id}_{M_2})(P_1)W_1^*\subset B_2\otimes
Q.$
Define $U_2=W_1^*.$ Then $C_2=U_2^*(\Phi_1\otimes {\rm
id}_{M_2})(P_1)M_2(B_2\otimes E_{2,2})(\Phi_1\otimes {\rm
id}_{M_2})(P_1)U_2.$ Moreover, $\psi_1(C_1)\subset C_2.$
Define ${\cal S}_2={\cal S}_{2,2}\cup \imath_2({\cal G}_2')\subset
C_2$ and ${\cal G}_3'={\cal G}_3\cup \theta_2({\cal S}_2).$ We may
assume that there is a finite subset ${\cal G}_3''\subset
A_3\otimes F_3$ such that ${\cal G}_3'\subset_{\ep_3/2} {\cal
G}_3''.$

We will use induction. Suppose that  unital monomorphism $\Psi_n:
B_n\otimes Q\to A_n\otimes Q,$ $\psi_n=\imath_{n+1}\circ
\Phi_n\circ \Psi_n: B_n\otimes Q\to B_{n+1}\otimes Q,$
$\theta_n=\Psi_n\circ \Phi_n: B_n\otimes Q\to A_{n+1}\otimes Q$
($n=1,2,...,N$), a unital finite dimensional \SCA\,
$E_{n,1}\subset Q$ ($n=1,2,...,N-1$), a projection $P_n\in
M_2(A_n\otimes E_{n,1})$ and a unitary $W_n\in M_2(A_n\otimes
E_{n,1})$ ($n=1,2,...,N$) have been constructed such that

\beq\label{MT-5}
\theta_n\circ \imath_n&\approx_{\ep_n}& \Phi_n\,\,\,{\rm on}\,\,\,
{\cal G}_n'\cup {\cal G}_n'',\\\label{MT-5+}
 \psi_n&=&\imath_{n+1}\circ \theta_n,
 \eneq
\beq\label{MT-5++}
W_n^* \Psi_n(B_n\otimes E_{n-1,n-1})W_{n}\subset P_nM_2(A_n\otimes
E_{n,1})P_n,
\eneq
where $\{E_{n,k}\}$ is an increasing sequence of unital finite
dimensional \SCA s of $Q$ whose union is dense in $Q,$ with
$E_{n-1, n-1}\subset E_{n,1},$ $E_{n,n}$ is a full matrix algebra
and
\beq\label{MT-5+++}
{d_n^3\over{{\rm rank}(E_{n,n})}}<1/2^n,
\eneq
where $\{{\cal S}_{n,k}\}$ is an increasing sequence of finite
 subsets of $B_n\otimes Q$ whose union is dense in $B_n\otimes Q,$
\beq\label{MT-6}
\cup_{j=1}^{n-1}\psi_{j,n-1}({\cal S}_{j,n})\cup \imath_n({\cal
G}_n'\cup {\cal G}_n'')\subset_{\ep_{n+1}/2} {\cal S}_{n,n},
\eneq
\beq\label{MT-7}
U_n{\cal S}_{n,k}U_n^*\subset P^{(n-1)}M_{2^{n-1}}(B_n \otimes
E_{n,k})P^{(n-1)},
\eneq
where $U_n=(\Phi_{n-1}\otimes {\rm
id}_{M_{2^{n-1}}})(W_{n-1}^*\otimes
1_{M_2})\psi_{n-1}^{(n-1)}(U_{n-1}),$ $n=3,4,...,N+1,$ where
$P^{(n-1)}\in M_{2^{n-1}}(B_n\otimes E_{n,n})$ with
$$
U_n^*P^{(n-1)}U_n=1_{B_n\otimes Q}=1_{B_n\otimes E_{n,n}},
$$
where ${\cal
 G}_{n+1}'={\cal G}_{n+1}\cup \theta_n({\cal S}_n),$
 ${\cal G}_{n+1}'\subset_{\ep_{n+1}/2} {\cal G}_{n+1}'',$
 ${\cal G}_{n+1}''\subset A_{n+1}\otimes F_{n+1}$ is a finite subset,
 $n=1,2,...,N+1$
 (with ${\cal G}_1'={\cal G}_1$), where
 $\psi_n^{(i)}=\psi_n\otimes {\rm id}_{M_{2^i}},$ $i=1,2,....$

Furthermore, $C_n=U_n^* P^{(n-1)}M_{2^{n-1}}(B_n\otimes
E_{n,n})P^{(n-1)}U_n$
and
$\psi_{n-1}(C_{n-1})\subset C_n,$ $n=2,3,..., N+1.$
By applying \ref{FL3}, we obtain
a unital monomorphism $\Psi_{N+1}:B_{N+1}\otimes Q\to
A_{N+1}\otimes Q$ such that
\beq\label{MT-8} \Psi_{N+1}\circ
\imath_{N+1}\approx_{\ep_{N+1}} {\rm id}_{B_{N+1}}\,\,\,{\rm
on}\,\,\, {\cal G}_{N+1}'\cup {\cal G}_{N+1}''.
\eneq
Moreover, there exists a unital finite dimensional
\SCA\,$E_{N+1,1}\subset Q,$ a projection $P_{N+1}\in
M_2(A_{N+1}\otimes E_{N+1,1})$ and a unitary $W_{N+1}\in
M_2(A_{N+1}\otimes E_{N+1,1})$ such that
\beq\label{MT-9}
W_{N+1}^*(\Psi_{N+1}( B_{N+1}\otimes E_{N+1,N+1}))W_{N+1}\subset
P_{N+1}M_2(A_{N+1}\otimes E_{N+1,1}) P_{N+1}.
\eneq
Define
$\theta_{N+1}=\Phi_{N+1}\circ \Psi_{N+1}: B_{N+1}\otimes Q\to
A_{N+2}\otimes Q$ and $\psi_{N+1}=\imath_{N+2}\circ \theta_{N+1}:
B_{N+1}\otimes Q\to B_{N+2}\otimes Q.$  Then, by(\ref{MT-5})
\beq\label{MT-10}
\theta_{N+1}\circ \imath_{N+1}
&\approx_{\ep_{N+1}}&\Phi_{N+1}\,\,\,{\rm on}\,\,\, {\cal G}_{N+1}'.
\eneq

Let $\{E_{N+1, k}\}$ be an increasing sequence of unital finite
dimensional \SCA s of $Q$ whose union is dense in $Q$ and $E_{N,
N}\subset E_{N+1,1},$ $E_{N+1, N+1}$ is a full matrix algebra with
\beq\label{MT-10+}
{d_{N+1}\over{{\rm rank}(E_{N+1,N+1})}}<1/2^{N+1}.
\eneq

 Let $\{{\cal S}_{N+1, k}\}$ be an increasing sequence of finite
subsets of $B_{N+1}\otimes Q$ whose union is dense in
$B_{N+1}\otimes Q.$ We may assume that
$$
\bigcup_{j=1}^N\psi_{j, N}({\cal S}_{j,n})\cup \imath_n({\cal
G}_n'\cup {\cal G}_n'')\subset_{\ep_{N+1}/2} S_{N+1, N+1}.
$$
Since $U_{N+1}^*P^{(N)}U_{N+1}=1_{B_{N+1}\otimes 1_Q},$ we may
further assume that
$$
U_{N+1}{\cal S}_{N+1,k}U_{N+1}^*\subset (\Phi_{N+1}\otimes {\rm
id}_{M_2})(P^{(N)})M_2(B_{N+1}\otimes E_{N+1,
k})(\Phi_{N+1}\otimes {\rm id}_{M_2})(P^{(N)}).
$$
Define   $U_{N+2}=(\Phi_{N+1}\otimes {\rm
id}_{M_{2^N}})(W_{N+1}^*\otimes
1_{M_{2^N}})\psi_{N+1}^{(N+1)}(U_{N+1}).$
Put ${\bar U}_{N+i}=(\psi_{N+i}\otimes {\rm
id}_{M_2})(U_{N+i}),\,i=1,2,$ and $\Phi_{m}^{(i)}=\Phi_{m}\otimes
{\rm id}_{M_{2^i}},m,i=1,2,...,$ and ${\bar
P}^{(N)}=\psi_{N+1}^{(N)}(P^{(N)}).$  Then (using (\ref{MT-5++}))
\beq\nonumber
\psi_{N+1}(C_{N+1})\\\nonumber
 &&\hspace{-1.1in}=\psi_{N+1}^{(N)}(U_{N+1}^*
P^{(N)}M_{2^N}(B_{N+1}\otimes E_{N+1,N+1})
 P^{(N)}U_{N+1})\\\nonumber
&&\hspace{-1.2in}={\bar U}_{N+1}^*{\bar P}^{(N)}
\psi_{N+1}^{(N)} (M_{2^N}(B_{N+1}\otimes E_{N+1, N+1})) {\bar
P}^{(N)}{\bar U}_{N+1}\\\nonumber
&&\hspace{-1.2in}\subset
{\bar U}_{N+1}^*{\bar P}^{(N)}
M_{2^N}(\Phi_{N+1}^{(2)}(W_{N+1}P_{N+1})M_2(B_{N+2}\otimes E_{N+2,
N+2})\Phi_{N+1}^{(2)}(P_{N+1}W_{N+1}^*)){\bar P}^{(N)}{\bar
U}_{N+1}
\eneq
Note that $ \Phi_{N+1}^{(2)}(W_{N+1}P_{N+1}W_{N+1}^*)$ has the
form
$$
\begin{pmatrix} 1_{B_{N+2}}\otimes 1_{E_{N+2, N+2}} & 0\\
                    0             & 0\end{pmatrix}.
                    $$

Therefore
\beq\label{ADD1}
{\bar P}^{(N)}\Phi_{N+1}^{(2)}(W_{N+1}P_{N+1}W_{N+1}^*)\otimes
1_{M_{2^N}}= {\bar P}^{(N)}.
\eneq
Put
\beq\label{ADD1+}                    {\bar W}_{N+1}=\Phi_{N+1}^{(N+1)}(W_{N+1}\otimes
                    1_{M_{2^{N}}})\andeqn
                    P^{(N+1)}={\bar W}_{N+1}^*{\bar P}^{(N)}{\bar W}_{N+1}.
                    \eneq
One has
\beq
\hspace{-0.3in}U_{N+2}^*P^{(N+1)}U_{N+2}
&=&\psi_{N+1}^{(N+1)}(U_{N+1})^*{\bar W}_{N+1} P^{(N+1)}
{\bar W}_{N+1}^*\psi_{N+1}^{(N+1)}(U_{N+1})\\
&=&\psi_{N+1}^{(N+1)}(U_{N+1})^*{\bar
P}^{(N)}\psi_{N+1}^{(N+1)}(U_{N+1})^*\\
&=&
\psi_{N+1}^{(N+1)}(U_{N+1}^*P^{(N)}U_{N+1})=\psi_{N+1}^{(N+1)}(1_{B_{N+1}\otimes
Q})=1_{B_{N+2}\otimes Q}.
\eneq
 It follows that (see (\ref{ADD1+}))
\beq
\psi_{N+1}(C_{N+1})\\
 &&\hspace{-0.5in}\subset  {\bar U}_{N+1}^*{\bar
 P}^{(N)}{\bar W}_{N+1}
M_{2^{N+1}}(B_{N+2}\otimes E_{N+2, N+2}){\bar W}_{N+1}^*{\bar
P}^{(N)}{\bar U}_{N+1}\\
&&\hspace{-0.5in}={\bar U}_{N+1}^*{\bar W}_{N+1}P^{(N+1)}
M_{2^{N+1}}(B_{N+2}\otimes E_{N+2, N+2})P^{(N+1)}{\bar
W}_{N+1}^*{\bar U}_{N+1}\\
&&\hspace{-0.5in}=U_{N+2}^* P^{(N+1)}M_{2^{N+2}}(B_{N+2}\otimes
E_{N+2, N+2})P^{(N+1)}U_{N+2}.
\eneq
Define
$$
C_{N+1}=U_{N+2}^* P^{(N+1)}M_{2^{N+2}}(B_{N+2}\otimes E_{N+2,
N+2})P^{(N+1)}U_{N+2}.
$$
Therefore
$$
\psi_{N+1}(C_{N+1})\subset C_{N+2}.
$$

 Put $E_n=A_n\otimes Q$ and $B_n'=B_n\otimes Q$ and define
$C=\lim_{n\to\infty}(B_n', \psi_n).$ It follows from (\ref{MT-5})
and (\ref{MT-5+}) that  the following diagram
$$
\begin{array}{ccccccc}
E_1 \hspace{0.2in}\stackrel{\Phi_1}{\to} & E_2
\,\,\,\,\,\stackrel{\Phi_2}{\to} & E_3
& \stackrel{\Phi_3}{\to}\cdots\\
\hspace{0.4in}\searrow{\imath_1}&
\hspace{-0.3in}\nearrow_{\theta_1}
&\hspace{-0.4in}\searrow{\imath_2}\nearrow_{\theta_2} &\hspace{-0.5in}\searrow{\imath_3}\\
&\hspace{-0.5in}B_1' \,\,\,\,\,\,\,\stackrel{\psi_1}{\to} &
\hspace{-0.1in}B_2' \,\,\,\,\,\stackrel{\psi_2}{\to} & B_3'
\cdots\\
\end{array}
$$
is approximately intertwining in the sense of Elliott. It follows
that $A\otimes Q\cong C.$ In particular, $C$ is a unital simple
\CA.

By (\ref{MT-7}),
$$
{\cal S}_{n,n}\subset C_n,\,\,\,n=1,2,....
$$
It follows from this and (\ref{MT-6}) that
$$
\overline{\cup_{n=1}^{\infty} \psi_{n,\infty}(C_n)}=C.
$$
Thus
$$
C=\lim_{n\to\infty}(C_n, (\psi_n)|_{C_n}).
$$
It follows from (\ref{MT-5+++}) that $C$ is a unital simple
AH-algebra with very slow dimension growth. It follows from a
theorem of Guihua Gong (see Theorem 2.5 of \cite{LnJFA}) that
$TR(C)\le 1.$

\end{proof}

Now we are ready to prove Theorem \ref{MT1}.

\begin{NN} The proof of Theorem \ref{MT1}:\end{NN}
\begin{proof}
It follows from \ref{MT3} that $A,\, B\in {\cal A}.$ Thus the
theorem follows from \ref{MTQ} (see \cite{Lnclasn}).

\end{proof}

\begin{NN}
{\rm  A unital simple inductive limit of generalized dimension
drop algebra $A$  is said to have no dimension growth, if
$A=\lim_{n\to\infty}(A_n, \phi_n),$ where
$$
A_n=\oplus_{j=1}^{m(n)} P_{n,j}{\bar D}_{n,j}P_{n,j},
$$
$$
{\bar D}_{n,j}=D_{k(n)}(X_{n,j}, Y_{n,1},Y_{n,2},...,Y_{n,s(n,j)},
m_{n,1}, m_{n,2},...,m_{n,s(n,j)}),
$$
$X_{n,j}$ is a finite dimensional connected compact metric space
with property (A), $\{Y_{n,i}:1\le i\le s(n,j)\}$ is a collection
of finitely many disjoint compact subsets of $X_{n,j},$
$P_{n,j}\in {\bar D}_{n,j}$ is a projection and $\{{\rm dim }(X_{n,j})\}$
is bounded.}
\end{NN}

\begin{NN} The proof of Theorem \ref{MT2}:\end{NN}
\begin{proof}
By \cite{W1}, ASH-algebras with no dimension growth are of finite decomposition rank. 
It follows from another  recent result of W. Winter (\cite{W3}) that $A$
and $B$ are ${\cal Z}$-stable. Therefore the theorem follows
from \ref{MT1} immediately.

\end{proof}

\bibliographystyle{amsalpha}
\bibliography{}

\begin{thebibliography}{BEE}


\bibitem{Dix}  J.  Dixmier, {\em On some $C\sp{*} $-algebras considered by Glimm},  J. Funct. Anal.
{\bf 1 },1967 182--203.

\bibitem{EGJS} G. A. Elliott, G.  Gong, Guihua, X. Jiang,  H.  Su,
\emph{A classification of simple limits of dimension drop $C\sp
\ast$-algebras}, Operator algebras and their applications
(Waterloo, ON, 1994/1995), 125--143, Fields Inst. Commun., {\bf
13}, Amer. Math. Soc., Providence, RI, 1997

\bibitem{EGL2} G. A. Elliott, G. Gong and L. Li \emph{On the classification of simple inductive limit
  $\mbox{C}^*$-algebras, $\mbox{II}$: The isomorphism theorem}, Invent. Math.
  \textbf{168} (2007), no.~2, 249--320.

\bibitem{JS}
X.~Jiang and H.~Su, \emph{On a simple unital projectionless
  $\textrm{C*}$-algebra}, Amer. J. Math \textbf{121} (1999), no.~2, 359--413.

\bibitem{Lnplms} H. Lin, {\em Tracial topological ranks of \CA s},
Proc. London Math. Soc., {\bf 83} (2001), 199-234.


\bibitem{Lnbk}
H.~Lin, \emph{An introduction to the classification of amenable
  $\textrm{C*}$-algebras}, World Scientific Publishing Co., River Edge, NJ,
  2001.


\bibitem{LnCorelle}
H. Lin,  \emph{Traces and simple $\textrm{C*}$-algebras with
tracial
  topological rank zero}, J. Reine Angew. Math. \textbf{568} (2004), 99--137.



\bibitem{LnJFA}
H. Lin, \emph{Simple nuclear $\mbox{C}$*-algebras of tracial
topological rank
  one}, J. Funct. Anal. \textbf{251} (2007), 601--679.

\bibitem{Lnhomp} H. Lin, {\em  Approximate homotopy of \hm s from $C(X)$ into a simple \CA\,},
preprint, arxiv.org/ OA/0612125.

\bibitem{Lnasym}
H. Lin,  \emph{Asymptotic unitary equivalence and asymptotically
inner
  automorphisms}, arXiv:math/0703610 (2007).

\bibitem{Lnclasn}
H. Lin, \emph{Asymptotically unitary equivalence and
classification of simple
  amenable $\textrm{C*}$-algebras}, arXiv: 0806.0636 (2008).

\bibitem{Lnapp}
H. Lin, \emph{Localizing the $\textrm{Elliott}$ conjecture at
strongly
  self-absorbing $\textrm{C*}$-algebras, $\textrm{II}$---an appendix}, arXiv:
  0709.1654v3 (2007).


\bibitem{LZ2} H. Lin and Z. Niu, {\em The range of a class of classifiable separable simple amenable C*-algebras},
Adv. Math. {\bf 219} (2008) 1729-1769.



\bibitem{Mb} G. J. Murphy, {\em $C\sp *$-Algebras and Operator Theory},. Academic Press, Inc., Boston, MA, 1990. x+286 pp. ISBN: 0-12-511360-9


\bibitem{M}
J.~Mygind, \emph{Classification of certain simple
$\textrm{C*}$-algebras with
  torsion $\mathrm{K}_1$}, Canad. J. Math. \textbf{53} (2001), no.~6,
  1223--1308.
\bibitem{S} H. Su,
{\em On the classification of $C\sp *$-algebras of real rank zero:
inductive limits of matrix algebras over non-Hausdorff graphs},
Mem. Amer. Math. Soc. {\bf 114} (1995), no. 547, viii+83 pp.

\bibitem{Ts}K.  Thomsen, {\em Limits of certain subhomogeneous $C\sp *$-algebras},
MŽm. Soc. Math. Fr. (N.S.)  {\bf 71} (1997), vi+125 pp.



\bibitem{T1} A. Toms, {\em On the classification problem for nuclear
C*-algebras}, Ann. of Math., to appear (arXiv:math/0509103).


\bibitem{TW} A. Toms and W. Winter, {\em $\cal Z$-stable ASH
algebras}, Canad. J. Math.  {\bf 60}   (2008),   703--720.


\bibitem{V1} J. Villadsen, {\em The range of the Elliott invariant of the simple
AH-algebras with slow dimension growth}.  $K$-Theory  {\bf 15}
(1998), 1--12.

\bibitem{V2} J. Villadsen, {\em On the stable rank of simple $C\sp
*$-algebras}, J. Amer. Math. Soc.  {\bf 12} (1999),  1091--1102.

\bibitem{W1} W.~Winter, \emph{ Decomposition rank of subhomogeneous $C\sp *$-algebras},  Proc. London Math. Soc.  {\bf 89}  (2004),   427--456.

\bibitem{W2}
W.~Winter, \emph{Localizing the $\textrm{Elliott}$ conjecture at
strongly
  self-absorbing $\textrm{C*}$-algebras}, arXiv: 0708.0283v3 (2007).
  
\bibitem{W3} W. ~Winter, \emph{Decomposition rank and Z-stability}, preprint (arXiv:0806.2948) (2008). 
  

\end{thebibliography}

\end{document}